\documentclass[11pt,a4paper]{amsart}
\usepackage[a4paper,margin=28mm]{geometry}
\usepackage{lmodern,microtype}
\usepackage{amssymb,mathtools,mathrsfs}
\usepackage{needspace}
\usepackage{booktabs,array,enumitem,aliascnt,xcolor,graphicx,placeins,longtable}
\usepackage[colorlinks=true,linkcolor=blue!55!black,citecolor=blue!55!black,urlcolor=blue!55!black]{hyperref}
\usepackage[nameinlink,capitalise,noabbrev]{cleveref}
\allowdisplaybreaks
\numberwithin{equation}{section}
\newtheorem{maintheorem}{Theorem}

\crefname{maintheorem}{Theorem}{Theorems}
\newtheorem{mainproposition}{Proposition}
\crefname{mainproposition}{Proposition}{Propositions}

\newaliascnt{lemma}{theorem}
\newtheorem{lemma}[lemma]{Lemma}
\aliascntresetthe{lemma}
\newaliascnt{proposition}{theorem}
\newtheorem{proposition}[proposition]{Proposition}
\aliascntresetthe{proposition}
\theoremstyle{remark}
\newaliascnt{remark}{theorem}
\newtheorem{remark}[remark]{Remark}
\aliascntresetthe{remark}
\newcommand{\R}{\mathbb R}
\newcommand{\Sph}{\mathbb S}

\newcommand{\tr}{\operatorname{tr}}
\newcommand{\dist}{\operatorname{dist}}
\newcommand{\Ric}{\operatorname{Ric}}
\newcommand{\eps}{\varepsilon}
\newcommand{\dd}{\,\mathrm d}
\DeclareMathOperator{\arccot}{arccot}
\DeclareMathOperator{\Sp}{Sp}
\DeclareMathOperator{\SO}{SO}

\makeatletter
\def\@settitle{\begin{center}%
  \baselineskip14\p@\relax\bfseries\@title
  \end{center}}
\makeatother
\begin{document}
\title[Lawlor's criterion and Hsiang-type minimizing cones]{Lawlor's curvature criterion and topologically trivial area-minimizing hypercones of Hsiang type}
\author{Hongbin Cui}
\address{Wu Wen-Tsun Key Laboratory of Mathematics, USTC, Chinese Academy of Sciences, School of Mathematical Sciences, University of Science and Technology of China, Hefei, Anhui 230026, China}
\email{cuihongbin@ustc.edu.cn}
\date{}
\hypersetup{pdftitle={Lawlor's curvature criterion and topologically trivial area-minimizing hypercones of Hsiang type},pdfauthor={Hongbin Cui}}
\begin{abstract}
For a closed embedded minimal hypersurface in a unit sphere, we show
that the existence of a uniform vanishing angle in Lawlor's criterion
suffices for its cone to minimize area. The normal-wedge condition
follows from Howard's rolling theorem. We derive explicit curvature
bounds and apply them to the reported Hsiang--Sterling curvature
estimates for topologically trivial cones. The resulting
area-minimizing cones give an affirmative
answer to the existence question raised by Ilmanen and White (\emph{Camb. J. Math.}, 2015).
\end{abstract}
\maketitle

\section{Introduction}

Let \(\Sigma^{k-1}\subset\Sph^k(1)\) be a smooth closed embedded minimal
hypersurface, and let \(C(\Sigma)\subset\R^{k+1}\) be its cone.
We call such a cone \emph{regular}: it is smooth away from the vertex,
as in \cite{HardtSimon85}.
After choosing an orientation, we call \(C(\Sigma)\) \emph{area minimizing}
if its unit-ball truncation \(C(\Sigma)\cap B_1(0)\) has least mass
among all compactly supported integral \(k\)-currents in \(\R^{k+1}\)
with boundary the oriented link \(\Sigma\).
By dilation, this agrees with local area minimization in the sense of
\cite[Definition~33.1]{Simon83}; see also
\cite[Proposition~2.5]{CuiChen26}.

Area-minimizing hypercones arise as tangent cones at singular points of
area-minimizing hypersurfaces. Their geometry plays an important role
in regularity theory, including recent advances in generic regularity
\cite{ChodoshMantoulidisSchulze26,ChodoshMantoulidisSchulze24}.
This motivates effective criteria for deciding when a minimal
hypercone minimizes area.
For algebraic minimal hypercones, algebraic subcalibration arguments
provide effective criteria for area minimization; see the author's
recent work \cite{CuiLaplacian26}.

For more general minimal hypercones, particularly those arising from
constructions of embedded minimal hypersurfaces in spheres, Lawlor's
curvature criterion \cite{Lawlor91} provides a geometric approach.
It combines a differential
inequality for the normal determinant with a geometric condition ensuring
that the corresponding normal wedges do not intersect. His criterion
has both pointwise and global formulations
\cite[Theorems~1.2.1 and~1.3.5]{Lawlor91}.
For embedded minimal hypersurfaces in a sphere, Howard's rolling theorem
\cite[Theorem~3]{Howard99} identifies the rolling radius of each
complementary domain with its minimum focal distance. Tan, Tang, Xie and
Yan also obtain this spherical consequence in the proof of
\cite[Lemma~5.1]{TanTangXieYan26}. We combine this geometric fact with
Lawlor's vanishing-angle construction to obtain a convenient curvature
criterion for area minimization.

Fix a unit normal field \(\nu\) along \(\Sigma\), and set
\(A=-\nabla\nu\). Write
\(\lambda_1(p),\ldots,\lambda_{k-1}(p)\) for the eigenvalues of \(A_p\),
that is, the principal curvatures counted with multiplicity, and set
\begin{equation}\label{eq:curvature-notation}
 K=\max_{p\in\Sigma}\max_{1\le i\le k-1}|\lambda_i(p)|,
 \qquad
 S=\max_{p\in\Sigma}|A_p|^2.
\end{equation}
Here the squared norm of the second fundamental form is
\[
 |A_p|^2=\tr(A_p^2)=\sum_{i=1}^{k-1}\lambda_i(p)^2.
\]
Both quantities are independent of the choice of unit normal.
A \emph{uniform vanishing angle} means one obtained from a common
Lawlor profile valid at every point of \(\Sigma\) and for both unit
normals; its precise formulation is recalled in
\cref{sec:criterion}.

\Needspace{5\baselineskip}
\begin{maintheorem}[Lawlor's criterion for regular minimal hypercones]
\label{main:criterion}
Let \(k\ge3\), and let
\(\Sigma^{k-1}\subset\Sph^k(1)\) be a smooth closed embedded minimal
hypersurface. If a uniform vanishing angle exists in Lawlor's curvature
criterion, then the oriented cone \(C(\Sigma)\) is locally area minimizing
among integral currents.
\end{maintheorem}

In our notation, this geometric identity is
\(\rho=F=\arccot K\), where \(\rho\) is the uniform two-sided tubular
radius and \(F\) is the minimum focal distance. The first normal-return
radius satisfies \(N\ge2F\). A uniform vanishing angle is at most \(F\),
so the normal-radius condition in Lawlor's criterion is automatic.
\Needspace{4\baselineskip}
For lower dimensions, the unparenthesized entries of Lawlor's
Table~1.4.1 can be used directly; for \(k\ge12\), his dimension
comparison gives a single explicit bound.

Combining \cref{main:criterion} with Lawlor's comparison proposition
\cite[Proposition~1.4.2, pp.~22--23]{Lawlor91} and the
\((k,\alpha^2)=(12,19)\) entry in his Table~1.4.1 gives the following
explicit consequence. We record the comparison argument in
\cref{lem:lawlor-comparison}.

\begin{mainproposition}[An explicit curvature bound]\label{prop:explicit}
Let \(k\ge12\), and let
\(\Sigma^{k-1}\subset\Sph^k(1)\) be a smooth closed embedded minimal
hypersurface. If
\begin{equation}\label{eq:explicit-bound}
 S\le B_1(k):=\frac{k^2}{8},
\end{equation}
then \(C(\Sigma)\) is locally area minimizing among integral currents.
\end{mainproposition}

Proposition~\ref{prop:explicit} may be viewed as an area-minimizing
counterpart of the classical curvature criterion for stability:
the radial Hardy inequality gives stability under
\(S\le(k-2)^2/4\); see \cite[Section~6.1]{Simons68}.
The same comparison with Lawlor's \((12,19)\) entry in fact gives
\(S\le19k^2/144\) for \(k\ge12\).
The next proposition improves this bound using an explicit polynomial
profile for Lawlor's exponential determinant estimate.

\begin{mainproposition}[A sharper curvature bound]\label{prop:improved}
Let \(k\ge24\), and let
\(\Sigma^{k-1}\subset\Sph^k(1)\) be a smooth closed embedded minimal
hypersurface. If
\begin{equation}\label{eq:improved-bound}
 S\le B_2(k):=\frac{11k^2}{72},
\end{equation}
then \(C(\Sigma)\) is locally area minimizing among integral currents.
In particular, the simpler bound \(S\le3k^2/20\) suffices.
\end{mainproposition}

The coefficient \(11/72\) improves Lawlor's bound
\(19/144\) by \(3/19\), or approximately \(15.79\%\).
The proof uses one polynomial profile in dimension \(24\), followed
by Lawlor's exponential dimension comparison.

We now give an application of Theorem A and the two curvature
propositions to cones over Hsiang's minimal hyperspheres and their
extensions. Hsiang \cite{Hsiang83I,Hsiang83II} constructed
non-equatorial embedded minimal hyperspheres. Hsiang and Sterling
\cite[Theorems~1 and~2]{HsiangSterling86} extended these constructions
and estimated the second fundamental form to establish stability for
many of the resulting cones.
Ilmanen and White \cite[p.~2]{IlmanenWhite15}
pointed out that the resulting cones are topologically trivial and
asked whether they minimize area.

Using the global curvature estimates reported by Hsiang and Sterling
\cite[Section~5 and Table~2]{HsiangSterling86}, we obtain the following
application of Proposition~\ref{prop:improved}.

\begin{maintheorem}[Hsiang-type cones]\label{main:hsiang}
Let \(\Sigma^{k-1}\subset\Sph^k(1)\) be one of the non-equatorial embedded
minimal hyperspheres listed in Hsiang--Sterling
\cite[Theorem~2 and Table~2]{HsiangSterling86} with one of the groups
\[
\begin{array}{ll}
 \SO(2)\times\SO(q),&q=19,\ldots,50,100,200,\\
 S(U(2)\times U(q)),&q=8,\ldots,20,50,100,\\
 \Sp(2)\times\Sp(q),&q=4,\ldots,20,25,50.
\end{array}
\]
Its oriented cone is a
nonflat, topologically trivial, regular area-minimizing hypercone
in \(\R^{k+1}\), with its only singularity at the vertex.
\end{maintheorem}

Here ``topologically trivial'' has the meaning used by Ilmanen and
White; these are the cones with embedded spherical links described
in \cite[p.~2]{IlmanenWhite15}.
These examples give an affirmative answer to their existence question.
By Table~\ref{tab:hs-data}, the selected values lie at least ten percent
below \(B_2(k)\), providing a substantial margin for this comparison.

\begin{remark}[Further applications]\label{rem:future-applications}
The cones in Hsiang--Sterling's stability theorem
\cite[Theorem~2]{HsiangSterling86} have dimension at least \(14\).
Refining Lawlor's vanishing profile, retaining the number and
multiplicities of principal curvatures in Lawlor's determinant estimate,
or sharpening global curvature estimates may extend the range of
Hsiang-type cones covered by Theorem~A.
Further potential applications arise from constructions using
equivariance or isoparametric foliations
\cite{LaiWei26,FiresterTsiamis26} and from the embedded minimal
hypersurfaces of Carlotto and Schulz \cite{CarlottoSchulz23};
Perdomo \cite{Perdomo26} studies the spectral properties and numerical
profile data of the latter.
\end{remark}

Section~\ref{sec:normal-geometry} recalls the normal-radius identities and
expands Howard's collision and index arguments in the spherical setting.
Section~\ref{sec:criterion} proves
\cref{main:criterion,prop:explicit,prop:improved}.
Section~\ref{sec:applications}
proves \cref{main:hsiang} and compares the Hsiang--Sterling data.

\section{Normal radii and the rolling theorem}
\label{sec:normal-geometry}

Let \(\Sigma^{k-1}\subset\Sph^k(1)\), \(k\ge3\), be smooth, closed,
embedded and minimal. It is connected by Frankel's intersection theorem
\cite{Frankel66}, and its complement consists of two connected domains
\(\Omega_+\) and \(\Omega_-\)
\cite[Section~1, p.~104]{HardtSimon85}.
Choose the unit normal \(\nu\) pointing into \(\Omega_+\), and write
\[
 E_\eps(p,t)=\cos t\,p+\eps\sin t\,\nu(p),
 \qquad \eps\in\{+1,-1\}.
\]
We distinguish the first return to \(\Sigma\), the first focal point,
and the loss of the nearest-point property along this normal geodesic.
All distances below are spherical distances on the unit sphere.
Distance to \(\Sigma\) agrees with intrinsic distance to the boundary
in either closed domain: a shortest spherical segment cannot leave
the domain before first meeting its boundary.
For background on second variation, focal points and the cut locus of
a point, see \cite[Chapter~1, Sections~7--9, and Chapter~5]{CheegerEbin08}.
The boundary cut-locus statements used below are the versions in
\cite[Propositions~2.1--2.2]{Howard99}.

\subsection{Return, focal and cut distances}
The first normal-return times and their global infimum are
\begin{equation}\label{eq:first-return}
 n_\eps(p)=\inf\{t\in(0,2\pi]:E_\eps(p,t)\in\Sigma\},
 \qquad N=\inf_{p,\eps}n_\eps(p).
\end{equation}
They satisfy \(0<n_\eps(p)<2\pi\). The smaller of \(n_+(p)\) and
\(n_-(p)\) is Lawlor's angular normal radius at the ray through \(p\)
\cite[Definition~1.1.3, p.~7]{Lawlor91}.
On each side, the infimum of the first-return times is Howard's
\emph{normal width} \cite[Definition~4.1]{Howard99}.
The geodesic is initially perpendicular to \(\Sigma\); perpendicularity
at its return point is not part of this definition.

For a principal vector \(e_i\in T_p\Sigma\), variation of the initial
point gives the normal Jacobi field
\begin{equation}\label{eq:normal-jacobi}
 J_i(t)=d_pE_\eps(e_i)
       =(\cos t-\eps\lambda_i(p)\sin t)e_i(t),
\end{equation}
where \(e_i(t)\) is parallel translation along the normal geodesic.
Thus a focal point occurs when the normal exponential map loses rank.
The spherical focal-point formula
\cite[Theorem~2.1(b), p.~11]{CecilRyan15} gives
\begin{equation}\label{eq:focal-radius}
 f_\eps(p)=\min_i\arccot\bigl(\eps\lambda_i(p)\bigr),
 \qquad F=\min_{p,\eps}f_\eps(p)=\arccot K.
\end{equation}
Here \(\arccot\) takes values in \((0,\pi)\), and
\(\arccot0=\pi/2\). Focal distance is therefore determined by the
second fundamental form at the initial point.

The normal cut times and the two-sided tubular radius are
\begin{equation}\label{eq:cut-radius}
 \begin{split}
 c_\eps(p)&=\sup\{a\in[0,n_\eps(p)]:
 \dist(E_\eps(p,t),\Sigma)=t\ \text{for }0\le t\le a\},\\
 \rho_\eps&=\inf_p c_\eps(p),\qquad
 \rho=\min\{\rho_+,\rho_-\}.
 \end{split}
\end{equation}
Before the cut time, \(p\) is the unique nearest point on \(\Sigma\).
At the cut time either a focal point is reached or another minimizing
normal segment reaches the same point
\cite[Proposition~2.1]{Howard99}. In particular,
\(c_\eps(p)\le f_\eps(p)\) and \(\rho\le F\).
The cut locus on the \(\eps\)-side is
\[
 \mathcal C_\eps
 =\{E_\eps(p,c_\eps(p)):p\in\Sigma\}.
\]
Howard's \emph{rolling radius} is
\begin{equation}\label{eq:rolling-radius}
 \operatorname{Roll}(\overline{\Omega_\eps})
 =\dist(\Sigma,\mathcal C_\eps)=\rho_\eps.
\end{equation}
For every \(0<r<\rho_\eps\), the spherical ball
\(B(E_\eps(p,r),r)\) lies in \(\Omega_\eps\) and is tangent to
\(\Sigma\) at \(p\). The same radius can be used at every boundary
point. This is the rolling-ball interpretation; it differs from the
radius of the largest ball that fits somewhere in the domain.

\begin{figure}[tbp]
 \centering
 \includegraphics[width=\textwidth]{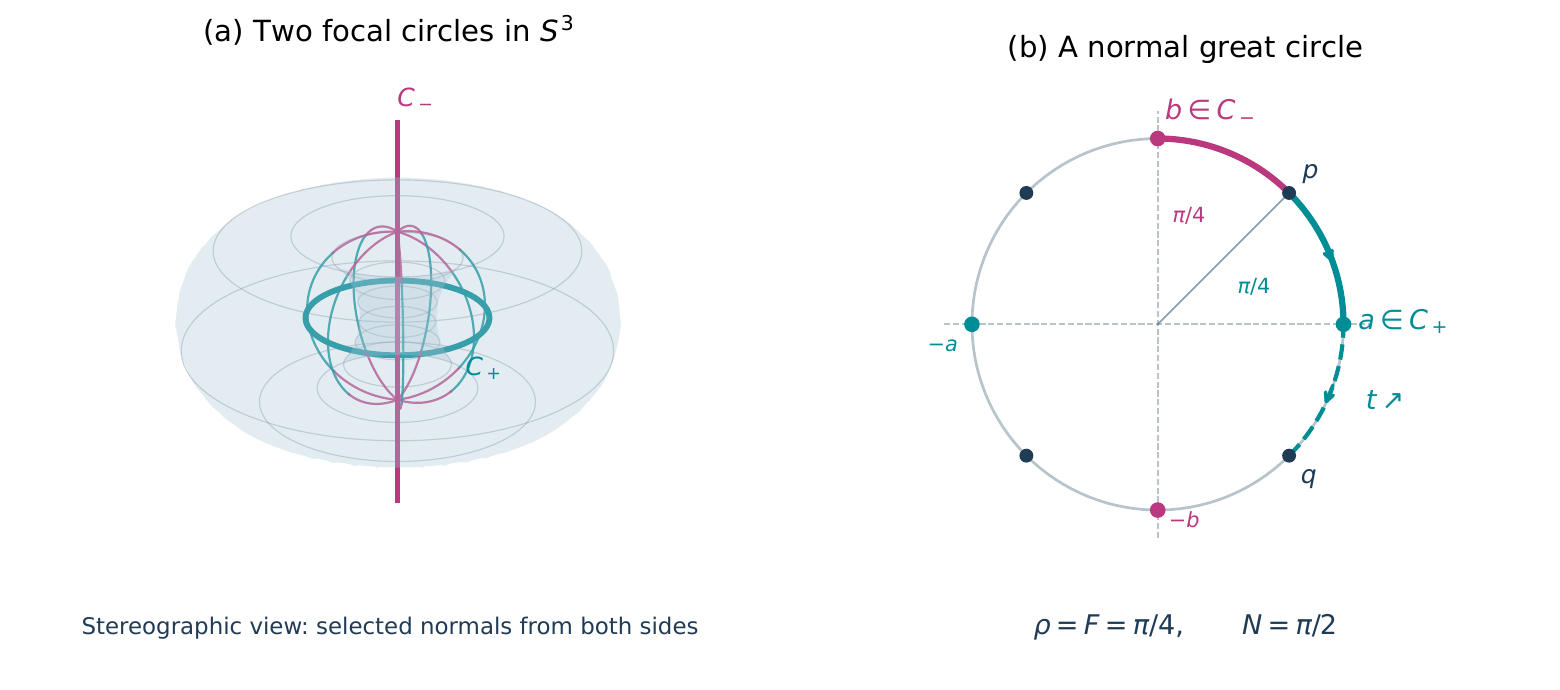}
 \caption{The minimal Clifford torus
 \(\Sigma=\{(x,y)\in\R^2\oplus\R^2:|x|=|y|=1/\sqrt2\}\subset\Sph^3\).
 Left: stereographic projection of the torus, selected normal geodesics,
 and its two focal circles \(C_+=\{(x,0):|x|=1\}\) and
 \(C_-=\{(0,y):|y|=1\}\). The latter passes through the projection
 pole and is displayed as a line. Distances are spherical, not Euclidean
 distances in this projection. Right: the normal great circle in
 \(\operatorname{span}\{a,b\}\), where \(a\in C_+\), \(b\in C_-\).
 Starting at \(p=(a+b)/\sqrt2\) with \(\nu=(a-b)/\sqrt2\),
 the positive normal reaches \(a\) at time \(\pi/4\) and first
 returns at \(q=(a-b)/\sqrt2\) at time \(\pi/2\).
 The principal curvatures are \(\pm1\), and
 \(c_\pm=f_\pm=\pi/4\), \(n_\pm=\pi/2\).
 Thus \(\rho=F=\pi/4\) and \(N=\pi/2\).
 The stereographic view is computed from the standard parametrization.
 For related visualizations, see the
 \href{https://virtualmathmuseum.org/Surface/clifford_torus/clifford_torus.html}{3DXM Virtual Math Museum}
 \cite{3DXMClifford} and \cite[Figure~1]{GarciaSotomayor09}.}
 \label{fig:clifford-radii}
\end{figure}

\subsection{The earliest collision and the index form}
We give the spherical version of Howard's collision and index arguments
\cite[Propositions~2.1--2.2 and proof of Theorem~4.3]{Howard99},
then derive the first-return bound used in Lawlor's criterion.

\begin{proposition}[A consequence of Howard's rolling theorem]
\label{prop:normal-radii}
The three global radii satisfy
\begin{equation}\label{eq:normal-radii}
 \rho=F=\arccot K,\qquad N\ge2F.
\end{equation}
Consequently, open normal wedges of any common angular radius at most
\(F\) are pairwise disjoint away from the origin.
\end{proposition}

\begin{proof}
\emph{Step 1: an earliest nonfocal collision has opposite terminal
velocities.}
Suppose that \(\rho<F\), and choose a side with rolling radius
\(\rho\). We first verify that this infimum is attained. Choose \(p_j\in\Sigma\)
with \(t_j=c_\eps(p_j)\to\rho\). For large \(j\), \(t_j<F\), so
the boundary cut-point alternative \cite[Proposition~2.1]{Howard99}
gives \(q_j\ne p_j\) with
\[
 z_j=E_\eps(p_j,t_j)=E_\eps(q_j,t_j),
 \qquad \dist(z_j,\Sigma)=t_j.
\]
By compactness of \(\Sigma\), pass to a subsequence with
\(p_j\to p\) and \(q_j\to q\). If \(p=q\), both pairs
\((p_j,t_j)\) and \((q_j,t_j)\) approach \((p,\rho)\).
Since \(\rho<F\), the normal exponential map is a local
diffeomorphism there, contradicting these distinct preimages.
Hence \(p\ne q\). Taking limits yields
\[
 z=E_\eps(p,\rho)=E_\eps(q,\rho),
 \qquad \dist(z,\Sigma)=\rho.
\]
Thus \(z\) is a cut point realizing the infimum, and there are
two distinct minimizing normal segments
\[
 \alpha_1,\alpha_2:[0,\rho]\longrightarrow\overline{\Omega_\eps},
 \qquad \alpha_1(\rho)=\alpha_2(\rho)=z,
\]
starting at distinct points \(p,q\in\Sigma\). Their lengths agree
because both realize \(\dist(z,\Sigma)\).

Set \(u_i=\alpha_i'(\rho)\). The absence of focal points gives smooth
local distance branches \(d_1,d_2\) near \(z\), corresponding to
small boundary patches near \(p,q\), with
\(\nabla d_i(z)=u_i\) and \(d_i(z)=\rho\).
The vectors \(u_1,u_2\) are distinct, since otherwise geodesic
uniqueness would give \(\alpha_1=\alpha_2\).
Hence \(H=\{d_1=d_2\}\) is a smooth hypersurface near \(z\).
If \(u_2\ne-u_1\), the nonzero vector \(w=-(u_1+u_2)\) satisfies
\begin{equation}\label{eq:collision-shortening}
 w\in T_zH,\qquad
 d d_1(w)=d d_2(w)=-1-\langle u_1,u_2\rangle<0.
\end{equation}
Moving from \(z\) in this direction along \(H\) therefore produces
two distinct normal segments of the same length less than \(\rho\)
with a common endpoint. This contradicts injectivity of the normal
exponential map before the uniform cut distance. Thus
\begin{equation}\label{eq:opposite-velocities}
 \alpha_2'(\rho)=-\alpha_1'(\rho).
\end{equation}
The same collinearity at an earliest collision is established in
Lawlor's nonintersection argument
\cite[proof of Lemma~1.3.1, pp.~15--16]{Lawlor91}.
The two segments consequently join to a smooth geodesic
\(\gamma:[0,L]\to\overline{\Omega_\eps}\), where \(L=2\rho\),
orthogonal to \(\Sigma\) at both ends; see
\cref{fig:collision}.

\begin{figure}[tbp]
 \centering
 \includegraphics[width=\textwidth]{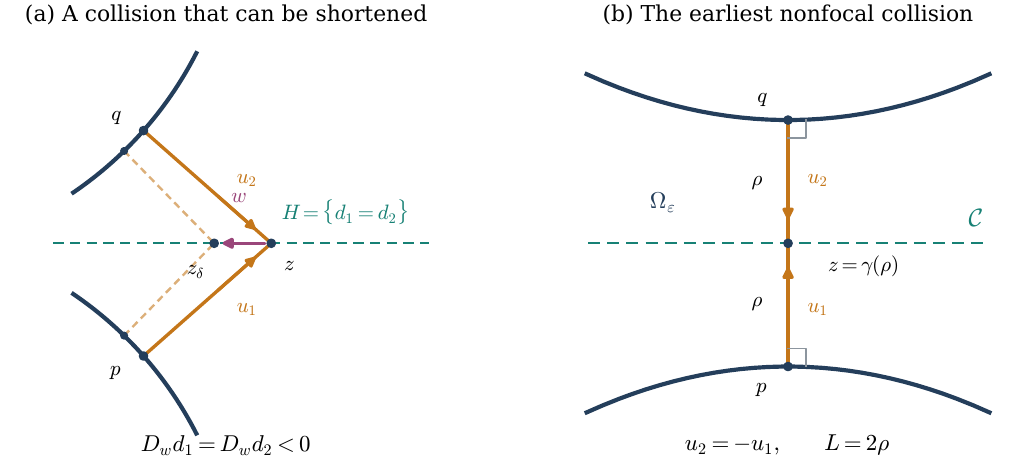}
 \caption{Local diagrams under the hypothetical assumption
 \(\rho<F\); the
 curved arcs represent boundary patches of the same hypersurface.
 Left: if the terminal unit vectors are not opposite, the direction
 \(w=-(u_1+u_2)\) is tangent to \(H=\{d_1=d_2\}\) and decreases
 both distance branches. Right: the two normals, each directed towards
 \(z\), have opposite terminal velocities. Reversing the second
 segment joins them to a smooth geodesic of length \(2\rho\),
 perpendicular to the boundary at \(p,q\). The dashed horizontal
 line represents the local cut locus, transverse to this geodesic.
}
 \label{fig:collision}
\end{figure}

\emph{Step 2: the joined segment is locally length minimizing with
endpoints free on the boundary.}
The same argument applied to any third minimizing segment to \(z\)
forces its terminal vector to equal \(u_2\); geodesic uniqueness then
identifies it with \(\alpha_2\). Thus there are exactly two nearest
boundary points. Compactness excludes other nearest boundary points
away from these two patches in a sufficiently small neighborhood \(U\) of
\(z\). Consequently,
\[
 \dist(x,\Sigma)=\min\{d_1(x),d_2(x)\}\qquad(x\in U),
\]
and the cut locus in \(U\) is the smooth hypersurface
\(H\cap U\) \cite[Proposition~2.1]{Howard99}.
Equation~\eqref{eq:opposite-velocities} shows that \(\gamma\)
crosses it transversely.
Every sufficiently close curve with endpoints on \(\Sigma\) must
also cross this local cut locus. If \(y\) is a crossing point, each
of its two portions has length at least \(\dist(y,\Sigma)\ge\rho\).
The total length is therefore at least \(2\rho=L\).
This is the local minimality assertion in
\cite[Proposition~2.2]{Howard99}. In particular, the index form is
nonnegative on variation fields tangent to \(\Sigma\) at the endpoints.

\emph{Step 3: the index form contradicts local length minimality.}
Let \(\eta\) be the inward unit normal of the chosen domain and let
\(A^\eta=-\nabla\eta\). Put \(T=\gamma'\); then
\(T(0)=\eta(p)\) and \(T(L)=-\eta(q)\).
For a smooth field \(V\perp T\) along \(\gamma\), tangent to
\(\Sigma\) at both ends, take a smooth variation whose endpoint
curves lie in \(\Sigma\) and whose variation field is \(V\).
The second variation of length is
\cite[proof of Theorem~4.3]{Howard99}
\begin{equation}\label{eq:boundary-index-form}
 \begin{split}
 I(V,V)={}&\int_0^L
 \bigl(|D_sV|^2-\langle R(V,T)T,V\rangle\bigr)\dd s\\
 &-\langle A^\eta_pV(0),V(0)\rangle
  -\langle A^\eta_qV(L),V(L)\rangle.
 \end{split}
\end{equation}
Both boundary terms have a minus sign. Indeed, the endpoint term in
the second variation is
\(\langle\nabla_VV,T\rangle|_0^L\), while the inward normals
are \(T(0)\) and \(-T(L)\). This gives precisely the two terms
displayed in \eqref{eq:boundary-index-form}.

Take an orthonormal parallel frame \(V_1,\ldots,V_{k-1}\) perpendicular
to \(T\). It is tangent to \(\Sigma\) at both endpoints. Since the
unit sphere has \(\Ric(T,T)=k-1\), summing
\eqref{eq:boundary-index-form} gives
\begin{equation}\label{eq:index-contradiction}
 \begin{split}
 0\le\sum_{i=1}^{k-1}I(V_i,V_i)
 &=-\int_0^L\Ric(T,T)\dd s
   -\tr A^\eta_p-\tr A^\eta_q\\
 &=-(k-1)L<0.
 \end{split}
\end{equation}
The boundary traces vanish by minimality. This contradiction proves
\(\rho=F\).

\begin{figure}[tbp]
 \centering
 \includegraphics[width=\textwidth]{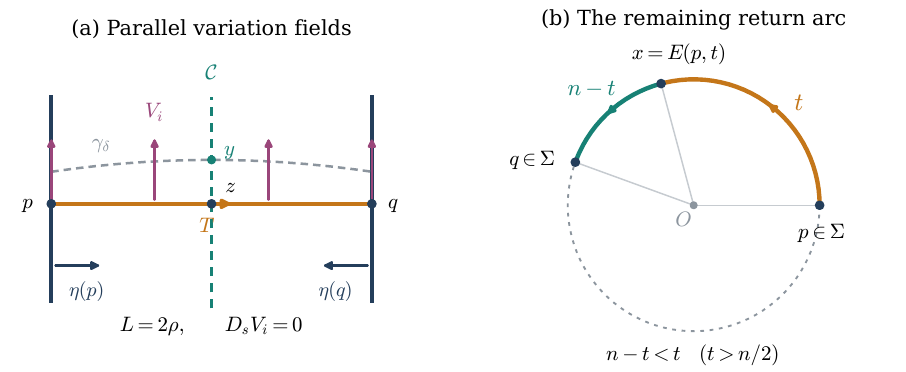}
 \caption{The two remaining geometric steps. Left: a nearby
 boundary-to-boundary curve \(\gamma_\delta\) crosses the local cut
 locus at \(y\); parallel
 fields along \(\gamma\) have admissible endpoint values. Their
 summed index is \(-(k-1)L\), in the spherical metric.
 Right: along a normal great circle, if \(q=E_\eps(p,n)\) is the
 first return and \(n/2<t<n\), the arc from \(E_\eps(p,t)\) to
 \(q\) has length \(n-t<t\). This proves \(c_\eps(p)\le n/2\).}
 \label{fig:index-return}
\end{figure}

\emph{Step 4: comparison with the first return and with normal wedges.}
If \(n=n_\eps(p)\) and \(n/2<t<n\), the remaining arc from
\(E_\eps(p,t)\) to \(E_\eps(p,n)\in\Sigma\) has length
\(n-t<t\). Therefore
\begin{equation}\label{eq:cut-return-comparison}
 c_\eps(p)\le\frac{n_\eps(p)}2,
 \qquad N\ge2\rho=2F.
\end{equation}
Finally, for \(0<a\le\rho\) define the open normal wedge
\[
 W_p(a)=\{r(\cos s\,p+\sin s\,\nu(p)):r>0,\ |s|<a\}.
\]
Every point of its spherical section has the unique nearest point
\(p\) on \(\Sigma\). Distinct wedges thus cannot intersect.
Taking \(a=F=\rho\) proves the last assertion.
\end{proof}

The identity \(\rho=F\) concerns global minima: it does not assert
\(c_\eps(p)=f_\eps(p)\) pointwise or identify the whole cut locus
with the focal set.

\FloatBarrier

\section{Lawlor's criterion and the comparison argument}
\label{sec:criterion}

For \(p\in\Sigma\) and \(\eps=\pm1\), put
\[
 P_{p,\eps}(t)=\det(I-t\eps A_p).
\]
In the profile formulation of Lawlor's construction
\cite[Lemma~2.3.6 and Theorem~2.3.8]{Lawlor91}, a common
vanishing profile is a nonincreasing function \(g\in C^1([0,T])\),
piecewise \(C^2\), such that
\begin{equation}\label{eq:profile-conditions}
 g(0)=1,\quad g'(0)=0,\quad
 g(t)>0\ (0\le t<T),\quad g(T)=0,
\end{equation}
and
\begin{equation}\label{eq:profile-inequality}
 \left(g(t)-\frac{t}{k}g'(t)\right)^2+
 \left(\frac{g'(t)}{k}\right)^2
 \le P_{p,\eps}(t)^2
 \qquad(p\in\Sigma,\ \eps=\pm1,\ 0\le t\le T).
\end{equation}
The associated common vanishing angle is \(\vartheta=\arctan T\).
The corresponding global lower envelope is
\[
 P_*(t)=\inf_{p\in\Sigma}\min_{\eps=\pm1}\det(I-t\eps A_p).
\]
On an interval where these determinants are nonnegative, the
right-hand side of \eqref{eq:profile-inequality} may be replaced by
\(P_*(t)^2\), or by the square of any nonnegative lower bound for
\(P_*(t)\). Lawlor's criterion then applies
when the open normal wedges of angle \(\vartheta\) are disjoint.
In the normal-radius formulation, it suffices that
\(\vartheta\le F\) and \(2\vartheta\le N\);
see \cite[Lemma~1.3.1 and Theorem~1.3.5]{Lawlor91}.

\begin{proof}[Proof of \cref{main:criterion}]
We first check that \(\vartheta\le F\).
If \(K=0\), then \(F=\pi/2>\arctan T\).
If \(K>0\), choose \(p,\eps\) so that \(\eps A_p\) has eigenvalue
\(K\), and set \(t_f=1/K\). Were \(T>t_f\), the right-hand side
of \eqref{eq:profile-inequality} would vanish at \(t_f\), forcing
\(g'(t_f)=g(t_f)=0\), contrary to
\eqref{eq:profile-conditions}. Hence
\[
 \vartheta=\arctan T\le\arctan(1/K)=F.
\]
This is the uniform version of Lawlor's focal-distance observation
\cite[p.~15]{Lawlor91}.
By \cref{prop:normal-radii}, the corresponding open normal wedges are
disjoint and \(2\vartheta\le2F\le N\). Lawlor's criterion proves the
area-minimizing conclusion.
\end{proof}

For comparison with the calibration formulation, extend \(g\) evenly
to \([-T,T]\) and use the coordinates
\[
 \Psi(r,p,t)=r(p+t\nu(p)),\qquad r>0.
\]
Writing \(\sigma\) for the pullback of the volume form of \(\Sigma\),
the associated form on the open tube is
\begin{equation}\label{eq:calibration-form}
 \omega=d\left(\frac{r^k}{k}g(t)\sigma\right),\qquad
 \|\omega\|_*^2=
 \frac{(g-tg'/k)^2+(g'/k)^2}{\det(I-tA_p)^2}.
\end{equation}
It restricts to the volume form of the oriented cone.
The extension across the wedge boundary and the area comparison are
provided by Lawlor's construction
\cite[Theorem~2.3.8 and Chapter~3]{Lawlor91}.

We use the following determinant estimate of Lawlor
\cite[Corollary~1.3.4]{Lawlor91}.

\begin{lemma}[Lawlor's determinant estimate]\label{lem:determinant}
Let \(A\) be a real symmetric matrix with \(\tr A=0\) and
\(|A|^2\le\alpha^2\), where \(\alpha>0\). Then
\begin{equation}\label{eq:exponential-bound}
 \det(I-tA)\ge E_\alpha(t):=(1-\alpha t)e^{\alpha t},
 \qquad 0\le t\le1/\alpha.
\end{equation}
\end{lemma}

Write
\[
 Q_k[g](t)=\left(g(t)-\frac{t}{k}g'(t)\right)^2
                    +\frac{g'(t)^2}{k^2}.
\]
The next lemma records the profile rescaling underlying Lawlor's
comparison proposition \cite[Proposition~1.4.2]{Lawlor91}.

\begin{lemma}[Dimension comparison]\label{lem:lawlor-comparison}
Let \(k\ge k_0\ge3\). Suppose a nonnegative nonincreasing profile \(g\)
satisfies \eqref{eq:profile-conditions}, with first zero
\(T_0\le1/\alpha_0\), and
\[
 Q_{k_0}[g](u)\le E_{\alpha_0}(u)^2
 \qquad(0\le u\le T_0).
\]
Then \(h(t)=g(kt/k_0)\) is a vanishing profile in dimension \(k\)
for the comparison parameter \(\alpha=k\alpha_0/k_0\).
\end{lemma}

\begin{proof}
Put \(u=kt/k_0\). Since \(g\ge0\) and \(g'\le0\),
\[
 Q_k[h](t)
 =\left(g(u)-\frac{u}{k}g'(u)\right)^2
                  +\frac{g'(u)^2}{k_0^2}
 \le Q_{k_0}[g](u)
 \le E_{\alpha_0}(u)^2
 =E_{k\alpha_0/k_0}(t)^2.
\]
The initial conditions are preserved, and the first zero of \(h\) is
\(k_0T_0/k\le1/\alpha\).
\end{proof}

Zhang \cite[Lemma~2.4 and Section~5.1]{Zhang26} also establishes a
dimension-comparison formula for the vanishing angles associated with
Lawlor's finite-dimensional determinant bound \(F(\alpha,t,m)\).

\begin{proof}[Proof of \cref{prop:explicit}]
Lawlor's Table~1.4.1 \cite[p.~21]{Lawlor91}, at
\((k_0,\alpha_0^2)=(12,19)\), gives the upper bound
\(11.23^\circ\) for the comparison vanishing angle; the table entries
are rounded upwards to the next hundredth of a degree
\cite[p.~19, note~(6)]{Lawlor91}. Thus there is a common comparison
profile for \(E_{\sqrt{19}}\) in dimension twelve, with vanishing angle
at most \(11.23^\circ\). By \cref{lem:lawlor-comparison}, a comparison
profile exists in dimension \(k\ge12\) with
\[
 \alpha=\frac{k\sqrt{19}}{12}.
\]
Since \(k^2/8<19k^2/144=\alpha^2\), assumption
\eqref{eq:explicit-bound} gives \(|\eps A_p|^2\le\alpha^2\)
for every \(p\) and both signs. Lemma~\ref{lem:determinant} makes this
profile valid uniformly on \(\Sigma\), so \cref{main:criterion} applies.
\end{proof}

To obtain further area-minimizing cones, we use the following explicit
polynomial profile. Combined with \cref{lem:lawlor-comparison}, it yields
the curvature bound in \cref{prop:improved}.

\begin{lemma}[An explicit vanishing profile]\label{prop:base-profile}
There is a nonnegative nonincreasing profile \(g\) satisfying
\eqref{eq:profile-conditions}, with first zero \(T_0=1/\sqrt{88}\), such that
\[
 Q_{24}[g](t)\le E_{\sqrt{88}}(t)^2
 \qquad(0\le t\le T_0).
\]
\end{lemma}

\begin{proof}
Set
\begin{equation}\label{eq:explicit-profile}
 G(u)=(1-u^2)^2
 \left(1-\frac{u^2}{4}+\frac{24u^3}{125}-\frac{47u^4}{500}\right),
 \qquad g(t)=G(\sqrt{88}\,t).
\end{equation}
On \([0,1]\), the last factor is at least \(82/125>0\).
For \(0<u<1\),
\[
 -\frac{G'(u)}{u(1-u^2)}
 =\frac92-\frac{72}{125}u-\frac{281}{250}u^2
       +\frac{168}{125}u^3-\frac{94}{125}u^4
 \ge\frac{256}{125}>0.
\]
Thus \(g\) is positive and decreasing before \(T_0\), with
\(g(0)=1\), \(g'(0)=0\), and \(g(T_0)=g'(T_0)=0\).

Write
\begin{align*}
 Q(u)&=\left(G(u)-\frac{uG'(u)}{24}\right)^2
                       +\frac{11}{72}G'(u)^2,\\
 R(u)&=(1-u)^2\sum_{j=0}^{8}\frac{(2u)^j}{j!}.
\end{align*}
Since \(R(u)\le(1-u)^2e^{2u}\) on \([0,1]\), it suffices to prove
\(Q\le R\). The quotient
\[
 P(u)=\frac{R(u)-Q(u)}{u^2(1-u)^2}
\]
is a polynomial of degree \(12\). Writing
\(P(u)=\sum_{i=0}^{12}p_i u^i\), its Bernstein coefficients are
\begin{equation}\label{eq:bernstein}
 \beta_j=\sum_{i=0}^j p_i\frac{\binom ji}{\binom{12}i},
 \qquad 0\le j\le12.
\end{equation}
Exact rational arithmetic gives
\[
 \min_{0\le j\le12}\beta_j
 =\frac{4760623}{264000000}>0.
\]
The Mathematica calculation in Appendix~\ref{app:certificate}
computes these thirteen coefficients and verifies the identity
\[
 P(u)=\sum_{j=0}^{12}\beta_j\binom{12}{j}u^j(1-u)^{12-j}.
\]
Each basis polynomial is nonnegative on \([0,1]\), so \(Q\le R\).
Substituting \(u=\sqrt{88}\,t\) proves the asserted inequality.
\end{proof}

\begin{proof}[Proof of \cref{prop:improved}]
Apply \cref{lem:lawlor-comparison,prop:base-profile} with
\(k_0=24\) and \(\alpha_0=\sqrt{88}\). This gives
\[
 \alpha_k=\frac{k\sqrt{88}}{24},\qquad
 h_k(t)=G(\alpha_k t),\qquad T_k=1/\alpha_k.
\]
Since \(S\le11k^2/72=\alpha_k^2\),
\cref{lem:determinant} makes \(h_k\) a common vanishing profile for
both normal directions, and \cref{main:criterion} applies.
In fact, the vanishing angle is strictly below the minimum focal distance:
minimality and Cauchy--Schwarz give
\[
 K^2\le\frac{k-2}{k-1}S<\alpha_k^2,
 \qquad \arctan T_k<F.
\]
\end{proof}

\begin{remark}[Pointwise vanishing angles]\label{rem:pointwise}
Lawlor's original criterion allows the vanishing angle to depend on
\(p\) \cite[Theorems~1.2.1 and~2.3.8]{Lawlor91}.
Writing \(c(p)=\min\{c_+(p),c_-(p)\}\), the condition
\(\vartheta(p)<c(p)\) at every point ensures that the corresponding
normal wedges are disjoint. Provided the pointwise profiles satisfy
Lawlor's admissibility and regularity hypotheses, his criterion then
gives area minimization. For such a family of profiles \(g(p,t)\), the calibration expression
\eqref{eq:calibration-form} has the same comass formula with
\(g'\) replaced by \(\partial_tg\), since
\(d_\Sigma g\wedge\sigma=0\); this cancellation is explicit in
\cite[p.~36]{Lawlor91}.
\end{remark}

\section{Applications to Hsiang-type cones}
\label{sec:applications}

\subsection{Topologically trivial cones and curvature comparison}

\begin{proof}[Proof of \cref{main:hsiang}]
Taking the reported global curvature estimates in
\cite[Section~5 and Table~2]{HsiangSterling86} as input,
Table~\ref{tab:hs-data} gives the bound required by
\cref{prop:improved}.
The link is a non-equatorial embedded hypersphere by
\cite[Theorems~1 and~2]{HsiangSterling86}.
Hence the cone is topologically trivial, as observed
by Ilmanen and White \cite[p.~2]{IlmanenWhite15}, and is smooth away
from the vertex. It is nonflat. If it were smooth at the vertex,
dilation invariance would make it a hyperplane. Thus the vertex is
its only singularity.
\end{proof}

For the three families in \cref{main:hsiang}, we take the type,
group, cone dimension \(k\), stability bound,
and reported curvature maximum \(\widehat S\) from
Hsiang--Sterling's Table~2
\cite[pp.~244--245]{HsiangSterling86}.
Their stability proofs combine numerical approximation with
truncation-error estimates
\cite[Section~5, pp.~237--241]{HsiangSterling86}.
The symbol \(\widehat S\) distinguishes the reported value from
the actual maximum \(S\) in \eqref{eq:curvature-notation}.
To omit near-threshold cases, we use the selection rule
\begin{equation}\label{eq:selection-margin}
 \widehat S\le\frac9{10}B_i(k).
\end{equation}
Table~\ref{tab:hs-data} retains only the configurations selected for
\(B_2\), adds both bounds, and records the relative reserve
\(\delta_2=1-\widehat S/B_2(k)\).
Bold bounds satisfy \eqref{eq:selection-margin}.
The accompanying data file preserves all \(82\) original entries.

{\footnotesize
\setlength{\tabcolsep}{3pt}
\begin{longtable}{@{}clrrrrrr@{}}
\caption{Selected entries of Hsiang--Sterling, Table~2. Bold bounds satisfy \(\widehat S\le0.9B_i(k)\). The last column gives \(\delta_2\) as a percentage. Bounds \(B_2\) and percentages are rounded down to three and one decimal places, respectively.}\label{tab:hs-data}\\
\toprule
\(d\) & Group & \(k\) & \((k-2)^2/4\) & \(\widehat S\) & Prop.~1 & Prop.~2 & \(\delta_2\) (\%)\\
\midrule
\endfirsthead
\multicolumn{8}{c}{\tablename~\thetable\ (continued)}\\
\toprule
\(d\) & Group & \(k\) & \((k-2)^2/4\) & \(\widehat S\) & Prop.~1 & Prop.~2 & \(\delta_2\) (\%)\\
\midrule
\endhead
\midrule
\multicolumn{8}{r}{Continued on next page}\\
\endfoot
\bottomrule
\endlastfoot
4 & \(\SO(2)\times\SO(19)\) & 38 & \(324\) & \(190.0\) & \(180.5\) & \(\mathbf{220.611}\) & \(13.8\)\\
4 & \(\SO(2)\times\SO(20)\) & 40 & \(361\) & \(200.3\) & \(200\) & \(\mathbf{244.444}\) & \(18\)\\
4 & \(\SO(2)\times\SO(21)\) & 42 & \(400\) & \(210\) & \(220.5\) & \(\mathbf{269.5}\) & \(22\)\\
4 & \(\SO(2)\times\SO(22)\) & 44 & \(441\) & \(221\) & \(242\) & \(\mathbf{295.777}\) & \(25.2\)\\
4 & \(\SO(2)\times\SO(23)\) & 46 & \(484^{\dagger}\) & \(231\) & \(\mathbf{264.5}\) & \(\mathbf{323.277}\) & \(28.5\)\\
4 & \(\SO(2)\times\SO(24)\) & 48 & \(529\) & \(241\) & \(\mathbf{288}\) & \(\mathbf{352}\) & \(31.5\)\\
4 & \(\SO(2)\times\SO(25)\) & 50 & \(576\) & \(252\) & \(\mathbf{312.5}\) & \(\mathbf{381.944}\) & \(34\)\\
4 & \(\SO(2)\times\SO(26)\) & 52 & \(625\) & \(262\) & \(\mathbf{338}\) & \(\mathbf{413.111}\) & \(36.5\)\\
4 & \(\SO(2)\times\SO(27)\) & 54 & \(676\) & \(273\) & \(\mathbf{364.5}\) & \(\mathbf{445.5}\) & \(38.7\)\\
4 & \(\SO(2)\times\SO(28)\) & 56 & \(729\) & \(283\) & \(\mathbf{392}\) & \(\mathbf{479.111}\) & \(40.9\)\\
4 & \(\SO(2)\times\SO(29)\) & 58 & \(784\) & \(293\) & \(\mathbf{420.5}\) & \(\mathbf{513.944}\) & \(42.9\)\\
4 & \(\SO(2)\times\SO(30)\) & 60 & \(841\) & \(303\) & \(\mathbf{450}\) & \(\mathbf{550}\) & \(44.9\)\\
4 & \(\SO(2)\times\SO(31)\) & 62 & \(900\) & \(313\) & \(\mathbf{480.5}\) & \(\mathbf{587.277}\) & \(46.7\)\\
4 & \(\SO(2)\times\SO(32)\) & 64 & \(961\) & \(324\) & \(\mathbf{512}\) & \(\mathbf{625.777}\) & \(48.2\)\\
4 & \(\SO(2)\times\SO(33)\) & 66 & \(1024\) & \(335\) & \(\mathbf{544.5}\) & \(\mathbf{665.5}\) & \(49.6\)\\
4 & \(\SO(2)\times\SO(34)\) & 68 & \(1089\) & \(345\) & \(\mathbf{578}\) & \(\mathbf{706.444}\) & \(51.1\)\\
4 & \(\SO(2)\times\SO(35)\) & 70 & \(1156\) & \(355\) & \(\mathbf{612.5}\) & \(\mathbf{748.611}\) & \(52.5\)\\
4 & \(\SO(2)\times\SO(36)\) & 72 & \(1225\) & \(366\) & \(\mathbf{648}\) & \(\mathbf{792}\) & \(53.7\)\\
4 & \(\SO(2)\times\SO(37)\) & 74 & \(1296\) & \(376\) & \(\mathbf{684.5}\) & \(\mathbf{836.611}\) & \(55\)\\
4 & \(\SO(2)\times\SO(38)\) & 76 & \(1369\) & \(386\) & \(\mathbf{722}\) & \(\mathbf{882.444}\) & \(56.2\)\\
4 & \(\SO(2)\times\SO(39)\) & 78 & \(1444\) & \(397\) & \(\mathbf{760.5}\) & \(\mathbf{929.5}\) & \(57.2\)\\
4 & \(\SO(2)\times\SO(40)\) & 80 & \(1521\) & \(407\) & \(\mathbf{800}\) & \(\mathbf{977.777}\) & \(58.3\)\\
4 & \(\SO(2)\times\SO(41)\) & 82 & \(1600\) & \(417\) & \(\mathbf{840.5}\) & \(\mathbf{1027.277}\) & \(59.4\)\\
4 & \(\SO(2)\times\SO(42)\) & 84 & \(1681\) & \(427\) & \(\mathbf{882}\) & \(\mathbf{1078}\) & \(60.3\)\\
4 & \(\SO(2)\times\SO(43)\) & 86 & \(1764\) & \(438\) & \(\mathbf{924.5}\) & \(\mathbf{1129.944}\) & \(61.2\)\\
4 & \(\SO(2)\times\SO(44)\) & 88 & \(1849\) & \(448\) & \(\mathbf{968}\) & \(\mathbf{1183.111}\) & \(62.1\)\\
4 & \(\SO(2)\times\SO(45)\) & 90 & \(1936\) & \(459\) & \(\mathbf{1012.5}\) & \(\mathbf{1237.5}\) & \(62.9\)\\
4 & \(\SO(2)\times\SO(46)\) & 92 & \(2025\) & \(469\) & \(\mathbf{1058}\) & \(\mathbf{1293.111}\) & \(63.7\)\\
4 & \(\SO(2)\times\SO(47)\) & 94 & \(2116\) & \(479\) & \(\mathbf{1104.5}\) & \(\mathbf{1349.944}\) & \(64.5\)\\
4 & \(\SO(2)\times\SO(48)\) & 96 & \(2209\) & \(490\) & \(\mathbf{1152}\) & \(\mathbf{1408}\) & \(65.1\)\\
4 & \(\SO(2)\times\SO(49)\) & 98 & \(2304\) & \(500\) & \(\mathbf{1200.5}\) & \(\mathbf{1467.277}\) & \(65.9\)\\
4 & \(\SO(2)\times\SO(50)\) & 100 & \(2401\) & \(511\) & \(\mathbf{1250}\) & \(\mathbf{1527.777}\) & \(66.5\)\\
4 & \(\SO(2)\times\SO(100)\) & 200 & \(9801\) & \(1028\) & \(\mathbf{5000}\) & \(\mathbf{6111.111}\) & \(83.1\)\\
4 & \(\SO(2)\times\SO(200)\) & 400 & \(39601\) & \(2065\) & \(\mathbf{20000}\) & \(\mathbf{24444.444}\) & \(91.5\)\\
\addlinespace[2pt]
4 & \(S(U(2)\times U(8))\) & 32 & \(225\) & \(134.4\) & \(128\) & \(\mathbf{156.444}\) & \(14\)\\
4 & \(S(U(2)\times U(9))\) & 36 & \(289\) & \(152.2\) & \(162\) & \(\mathbf{198}\) & \(23.1\)\\
4 & \(S(U(2)\times U(10))\) & 40 & \(361\) & \(169.7\) & \(\mathbf{200}\) & \(\mathbf{244.444}\) & \(30.5\)\\
4 & \(S(U(2)\times U(11))\) & 44 & \(441\) & \(187\) & \(\mathbf{242}\) & \(\mathbf{295.777}\) & \(36.7\)\\
4 & \(S(U(2)\times U(12))\) & 48 & \(529\) & \(205\) & \(\mathbf{288}\) & \(\mathbf{352}\) & \(41.7\)\\
4 & \(S(U(2)\times U(13))\) & 52 & \(625\) & \(223\) & \(\mathbf{338}\) & \(\mathbf{413.111}\) & \(46\)\\
4 & \(S(U(2)\times U(14))\) & 56 & \(729\) & \(240\) & \(\mathbf{392}\) & \(\mathbf{479.111}\) & \(49.9\)\\
4 & \(S(U(2)\times U(15))\) & 60 & \(841\) & \(258\) & \(\mathbf{450}\) & \(\mathbf{550}\) & \(53\)\\
4 & \(S(U(2)\times U(16))\) & 64 & \(961\) & \(276\) & \(\mathbf{512}\) & \(\mathbf{625.777}\) & \(55.8\)\\
4 & \(S(U(2)\times U(17))\) & 68 & \(1089\) & \(294\) & \(\mathbf{578}\) & \(\mathbf{706.444}\) & \(58.3\)\\
4 & \(S(U(2)\times U(18))\) & 72 & \(1225\) & \(311\) & \(\mathbf{648}\) & \(\mathbf{792}\) & \(60.7\)\\
4 & \(S(U(2)\times U(19))\) & 76 & \(1369\) & \(329\) & \(\mathbf{722}\) & \(\mathbf{882.444}\) & \(62.7\)\\
4 & \(S(U(2)\times U(20))\) & 80 & \(1521\) & \(347\) & \(\mathbf{800}\) & \(\mathbf{977.777}\) & \(64.5\)\\
4 & \(S(U(2)\times U(50))\) & 200 & \(9801\) & \(876\) & \(\mathbf{5000}\) & \(\mathbf{6111.111}\) & \(85.6\)\\
4 & \(S(U(2)\times U(100))\) & 400 & \(39601\) & \(1764\) & \(\mathbf{20000}\) & \(\mathbf{24444.444}\) & \(92.7\)\\
\addlinespace[2pt]
4 & \(\Sp(2)\times\Sp(4)\) & 32 & \(225\) & \(121.3\) & \(128\) & \(\mathbf{156.444}\) & \(22.4\)\\
4 & \(\Sp(2)\times\Sp(5)\) & 40 & \(361\) & \(153.6\) & \(\mathbf{200}\) & \(\mathbf{244.444}\) & \(37.1\)\\
4 & \(\Sp(2)\times\Sp(6)\) & 48 & \(529\) & \(186\) & \(\mathbf{288}\) & \(\mathbf{352}\) & \(47.1\)\\
4 & \(\Sp(2)\times\Sp(7)\) & 56 & \(729\) & \(218\) & \(\mathbf{392}\) & \(\mathbf{479.111}\) & \(54.4\)\\
4 & \(\Sp(2)\times\Sp(8)\) & 64 & \(961\) & \(251\) & \(\mathbf{512}\) & \(\mathbf{625.777}\) & \(59.8\)\\
4 & \(\Sp(2)\times\Sp(9)\) & 72 & \(1225\) & \(283\) & \(\mathbf{648}\) & \(\mathbf{792}\) & \(64.2\)\\
4 & \(\Sp(2)\times\Sp(10)\) & 80 & \(1521\) & \(315\) & \(\mathbf{800}\) & \(\mathbf{977.777}\) & \(67.7\)\\
4 & \(\Sp(2)\times\Sp(11)\) & 88 & \(1849\) & \(348\) & \(\mathbf{968}\) & \(\mathbf{1183.111}\) & \(70.5\)\\
4 & \(\Sp(2)\times\Sp(12)\) & 96 & \(2209\) & \(380\) & \(\mathbf{1152}\) & \(\mathbf{1408}\) & \(73\)\\
4 & \(\Sp(2)\times\Sp(13)\) & 104 & \(2601\) & \(413\) & \(\mathbf{1352}\) & \(\mathbf{1652.444}\) & \(75\)\\
4 & \(\Sp(2)\times\Sp(14)\) & 112 & \(3025\) & \(445\) & \(\mathbf{1568}\) & \(\mathbf{1916.444}\) & \(76.7\)\\
4 & \(\Sp(2)\times\Sp(15)\) & 120 & \(3481\) & \(477\) & \(\mathbf{1800}\) & \(\mathbf{2200}\) & \(78.3\)\\
4 & \(\Sp(2)\times\Sp(16)\) & 128 & \(3969\) & \(510\) & \(\mathbf{2048}\) & \(\mathbf{2503.111}\) & \(79.6\)\\
4 & \(\Sp(2)\times\Sp(17)\) & 136 & \(4489\) & \(542\) & \(\mathbf{2312}\) & \(\mathbf{2825.777}\) & \(80.8\)\\
4 & \(\Sp(2)\times\Sp(18)\) & 144 & \(5041\) & \(575\) & \(\mathbf{2592}\) & \(\mathbf{3168}\) & \(81.8\)\\
4 & \(\Sp(2)\times\Sp(19)\) & 152 & \(5625\) & \(607\) & \(\mathbf{2888}\) & \(\mathbf{3529.777}\) & \(82.8\)\\
4 & \(\Sp(2)\times\Sp(20)\) & 160 & \(6241\) & \(640\) & \(\mathbf{3200}\) & \(\mathbf{3911.111}\) & \(83.6\)\\
4 & \(\Sp(2)\times\Sp(25)\) & 200 & \(9801\) & \(801\) & \(\mathbf{5000}\) & \(\mathbf{6111.111}\) & \(86.8\)\\
4 & \(\Sp(2)\times\Sp(50)\) & 400 & \(39601\) & \(1609\) & \(\mathbf{20000}\) & \(\mathbf{24444.444}\) & \(93.4\)\\
\end{longtable}
}

\noindent\(^{\dagger}\) In the \(\SO(2)\times\SO(23)\) row,
the stability formula gives \(484\) in place of the printed \(448\).

Within these three families, the same ten-percent reserve selects
\(61\), \(64\), and \(68\) configurations for \(B_1(k)\),
Lawlor's \(19k^2/144\), and \(B_2(k)\), respectively.
Relative to Lawlor's bound, \cref{prop:improved} adds
\(\SO(2)\times\SO(q)\) for \(19\le q\le21\) and
\(S(U(2)\times U(8))\).

\section*{Acknowledgments}
This work was supported by the National Natural Science Foundation of China
(No.\ 12301068), the project of Stable Support for Youth Team in Basic
Research Field, CAS (YSBR-001), and the Xiaomi Young Scholar Fund.

\medskip
\noindent\textbf{AI disclosure.}
ChatGPT (OpenAI) was used to assist with language editing, literature searches,
symbolic and algebraic computations, and AI-assisted figure generation based
on the author's sketches and instructions. The author takes full responsibility
for the mathematical content and computational conclusions.

\clearpage
\appendix
\section{Mathematica calculation for Lemma~\ref{prop:base-profile}}
\label{app:certificate}
\begin{center}
\includegraphics[width=\textwidth,trim=55bp 285bp 44bp 121bp,clip]{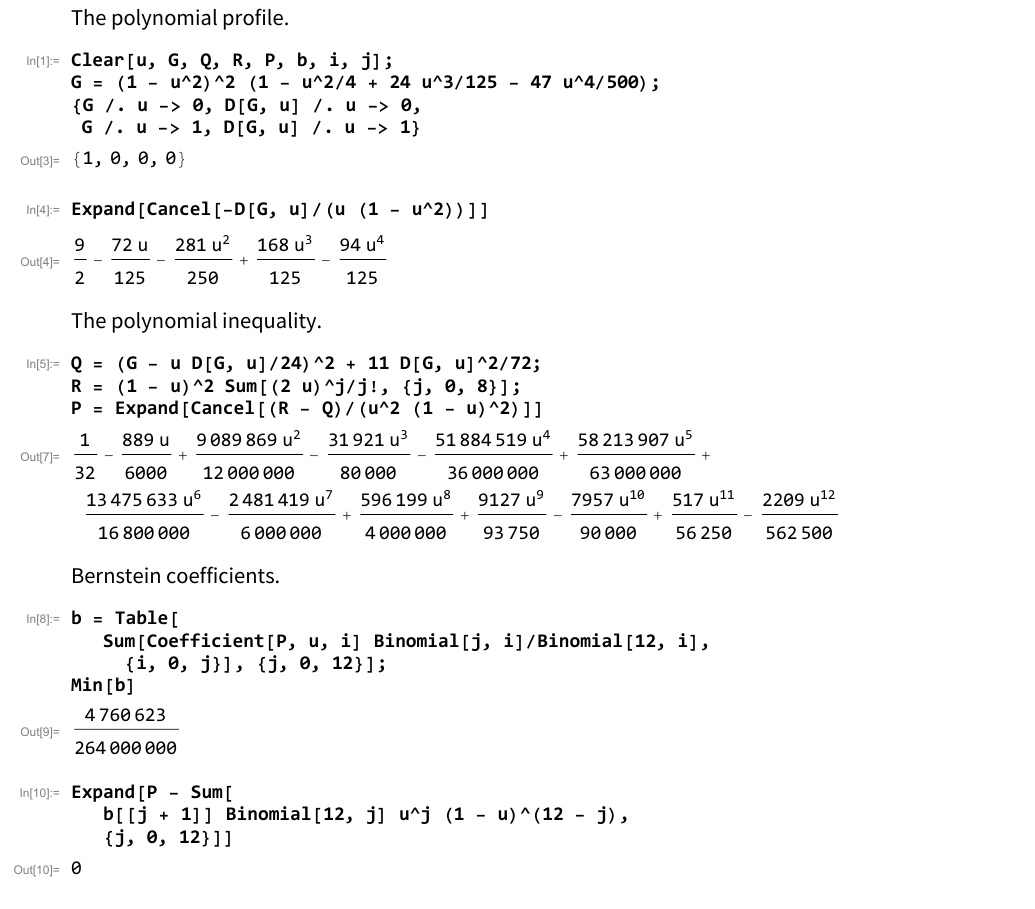}
\end{center}
\end{document}